\documentclass[letterpaper, 10 pt, conference]{ieeeconf}  

\IEEEoverridecommandlockouts                              

\newtheorem{theorem}{Theorem}

\newtheorem{define}{Definition}

\newtheorem{remark}{Remark}

\newtheorem{assumption}{Assumption}

\usepackage{graphics} 
\usepackage{epsfig} 
\usepackage{xcolor}
\usepackage{amsmath} 
\usepackage{amssymb}  
\usepackage[dvipsnames]{xcolor}

\usepackage{stmaryrd} 
\usepackage{bbm} 
\usepackage{algorithmic}
\usepackage[linesnumbered,ruled,vlined]{algorithm2e}
\usepackage{url}
\usepackage{dsfont}
\usepackage{caption}
\usepackage{subcaption}
\usepackage{booktabs}
\usepackage{dsfont}
\usepackage{mathtools}
\usepackage{blkarray}
\usepackage[normalem]{ulem}

\usepackage{tikz}
\usetikzlibrary{arrows,positioning}

\usepackage{dsfont}

\usepackage{tikz}
\usetikzlibrary{shapes,arrows,positioning}
\usetikzlibrary{automata,arrows,positioning,calc}
\usetikzlibrary{arrows.meta,positioning}
\usetikzlibrary{
  arrows.meta,
  positioning,
  decorations.pathreplacing,
  calc
}

\newcommand{\bsalpha}{\boldsymbol{\alpha}}

\newcommand{\RR}{\mathbb{R}}

\newcommand{\bLambda}{\mathbf{\Lambda}}

\newcommand\cI{\mathcal I}
\newcommand\cK{\mathcal K}

\newcommand\cM{\mathcal M}

\newcommand\bZ{\boldsymbol{Z}}

\DeclareMathOperator*{\argmin}{arg\,min}

\newcommand\rmS{\mathrm{S}}
\newcommand\rmK{\mathrm{K}}
\newcommand\rmI{\mathrm{I}}
\newcommand\rmR{\mathrm{R}}

\newcommand{\cJ}{\mathcal{J}}

\makeatletter
\let\NAT@parse\undefined
\makeatother
\usepackage[colorlinks]{hyperref}

\title{\LARGE \bf
Regulation of Rumor Propagation via \\(Multi-Leader) Stackelberg Graphon Games
}

\author{Huaning Liu and G\"ok{\c c}e Dayan{\i}kl{\i}
\thanks{Department of Statistics,
  University of Illinois at Urbana-Champaign, 
  Champaign, IL 61820, USA 
        {\tt\small huaning3@illinois.edu}}
\thanks{Department of Statistics,
  University of Illinois at Urbana-Champaign, 
  Champaign, IL 61820, USA 
        {\tt\small gokced@illinois.edu}}
}

\begin{document}

\maketitle
\thispagestyle{empty}
\pagestyle{empty}

\begin{abstract}

We study the control of rumor propagation in large networked populations by using Stackelberg graphon games. We first introduce a principal who wants to incentivize the spread of her preferred news and discourage the spread of non-preferred news. We define the Stackelberg graphon game equilibrium (SGGE), characterize the graphon game Nash equilibrium (GGNE) with a forward-backward differential equation system, and establish existence results. We further formulate a multi-leader model with two competing principals, each incentivizing her own preferred news. Finally, we propose a bi-level algorithm for computing (multi-leader) Stackelberg graphon game equilibria and conclude with numerical experiments where we show that existence of competing principals will result in strong opinion divisions in the population.

\end{abstract}

\section{Introduction}
Understanding and controlling the spread of rumors in large populations is a critical challenge with implications for public policy, media regulation, and social dynamics. Unlike epidemic spread models, rumor propagation involves strategic decision-making and belief reversals, often influenced by external actors such as governments or media platforms. These higher level actors which we refer to as principals can shape public opinion by incentivizing individuals to amplify preferred narratives or suppress undesirable ones through financial rewards or allocation of advertisement budgets.
 
Mean field games (MFGs) offer an approximation framework for large population non-cooperative games under the assumption of homogeneous agents who interact symmetrically~\cite{carmona_MFG_booki} \cite{lasry_lions_MFG}. It leverages the mean-field approximation by focusing on the interactions between the representative agent and the population distribution (of states and/or controls), effectively characterizing the equilibrium by a coupled forward backward differential equation system. In the literature, MFGs have been utilized to model many real life applications, including market competition via advertisement~\cite{duopoly_twoleader,SALHAB20221079,advertisement_carryover}, energy management~\cite{sircar_chan_energy,hubert_energy,alasseur2020extended,dayanikli_acc2024,carbon_StackelbergMFG,dayanikli_energy_acc2025}, systemic risk~\cite{carmona_sys_risk}, etc. On the other hand, \textit{graphon games} extend assumptions of MFG by introducing asymmetric interactions among heterogeneous agents via a graphon (limit of a dense graph which will be introduced in later sections) \cite{graphon_parise,caines_graphon,carmona2019stochasticgraphongamesi,aurell2022stochastic}. Therefore, they are particularly useful for modeling the games involving network interactions such as opinion dynamics~\cite{rumor_skir} and epidemic control~\cite{graphon_epidemics}.

Classical rumor models such as Daley-Kendall (DK)~\cite{daley-kendall} and Maki-Thompson (MT)~\cite{Maki1973MathematicalMA}, and their extensions to incorporate additional states~\cite{rumor_eg_1_yu,rumor_eg_2_xiao} that are supported by sociological justifications treat rumor spread analogously to epidemic spread, but they lack mechanisms for strategic individual level decision-making and external influence by the regulators. In contrast, our model focuses on incorporating both of these aspects.

In opinion dynamics applications in large populations, individuals often face a tradeoff between reputation and risk, transforming the public opinion into a strategic game~\cite{marie_game_opinion}. To analyze such problems of individual decision-making, some recent studies applied the concept of mean-field interactions (see e.g. continuous state analysis in~\cite{basar_rumor_MFG}, and finite state nonlinear case in~\cite{rumor_skir}). However, these models did not cover the \textit{optimal} policies for the regulators to pull the opinion in the population towards their preferred direction. In this way, one of the main goals of studying rumor propagation models is to provide guidance for government or leaders' decision-making. Economics literature has studied principal-agent and multi-agent-principal problems extensively see e.g. \cite{principal-agent-1, principal-agent-2}. Building on the contract theory literature, recent works consider multi-agent-principal problem under large agent populations with mean-field approximations (also termed as \textit{Stackelberg mean field games}), and provide linear-quadratic (LQ) case analysis \cite{StackelbergMFG-Delayed,LQStackelbergMFG}, convergence analysis \cite{djete2023stackelbergmeanfieldgames},numerical methods \cite{ML-penalty-stackelbergMFG,stackelbergMFGEpidemics}, and multi-leader case under the static time setting~\cite{duopoly_twoleader}.

In this work, we consider non-cooperative individuals interacting on a \textit{network} and leaders incentivizing the individuals to pursue their own interests. We will discuss both the cases of single-leader and duo-leader, where in the latter case leaders are also competing. Our contributions are three fold. First, we introduce \textit{single-principal Stackelberg graphon games} and \textit{duo-principal Stackelberg graphon games} for rumor propagation and opinion dynamics control for large population of rational agents. To the best of our knowledge, this is the first work on Stackelberg graphon games on the finite state setting. Second, we provide forward-backward ordinary differential equation (FBODE) system that characterizes the graphon game equilibrium of the agents in the population given the regulations by principals and provide the existence results. Finally, we give numerical results for both single and multi-leader settings. We note that prior work examined a similar setting \cite{mfg_opinion_dy}, where users sharing common characteristics were grouped into multiple populations under mean-field interactions. Our work considers distinguishable individuals with finite states, and we remark that the example of a piecewise-constant graphon in Section~\ref{sec:numerics} reflects a similar idea of multi-population.

The paper is organized as follows. In Section~\ref{sec:model_single}, we state and justify the details of the single-leader Stackelberg graphon game model of rumor propagation control and define the related equilibria. In Section~\ref{sec:theory}, we give the main theoretical results which includes the characterization and the existence of the graphon game equilibrium. In Section~\ref{sec:model_multi}, we introduce the multi-leader Stackelberg graphon game model and the correponding equilibrium. In Section~\ref{sec:numerics}, we provide our numerical algorithm and provide numerical examples with power-law graphon and piece-wise constant graphon. We give the conclusion and future directions in~\ref{sec:conclusion-future}.

\section{Single-Principal Model}
\label{sec:model_single}
The authors’ earlier work~\cite{rumor_skir} analyzed minor-agent interactions under rumor model equipped with graphon without the presence of principals and their effects on the minor agents' models, which sets foundation for characterizing minor agents’ behavior. In this section, we extend the minor agent model to a setting to include principal(s) (i.e., leader(s)) in their model. Then, we formulate and explain the principal’s problem. The preliminary construction of general form of finite-player game is deferred to~\ref{sec:finite_player} for brevity.

\subsection{Minor Agents' Model}
We consider a continuum of heterogeneous (minor) agents interacting with a principal. Agents interact with each other on a network and each agent has negligible influence on the population due to the infinite population size. Agents are impacted by some incentives or regulations from the principal(s) i.e., the graphon game Nash equilibrium is partially determined by the decisions of principal.

\par{\textbf{States and transitions. }}
The minor agents transition among four states: \textit{Uninformed} (S), \textit{Known} (K), \textit{Infected} (I), and \textit{Uninterested} (R), whose transition flows are illustrated in the diagram~\ref{fig:transition}. It extends the classic finite-state rumor models (DK and MT) by including an anti-rumor state K, first proposed in~\cite{rumor_eg_2_xiao}. States K and I play opposite positions, corresponding to individuals who are actively spreading the principal’s preferred news and non-preferred news respectively. In particular, interpreting preferred news as truth and non-preferred news as fake news recovers the illustration given in earlier work~\cite{rumor_skir}.

Uninformed agents (S) may transition into believing the preferred news (K) or the non-preferred news (I), and individuals in these opposing states can switch opinions to the other. The state transition rates are jointly determined by the socialization levels of other agents, the principal’s incentives, underlying network structure, and the model parameters. On the other hand, agents in state K and I are getting uninterested (R) after an exponentially distributed time with rate $\mu_{\mathbf{K}}\ (\text{resp. } \mu_{\mathbf{I}}) \in(0, \bar{\mu}]$ for some $\bar{\mu}>0$. Individuals in state R revert to uninformed (S) with rate $\eta \in [0, \bar{\eta}]$, capturing a forgetting process. In addition, let $\beta: E \rightarrow [0, \bar{\beta}]$ be the \textit{base meeting intensity} of individuals, where $E := \{\rmS, \rmK, \rmI, \rmR\}$ is the finite state space and $\bar\beta>0$ is a constant uniform bound.
We remark that $\beta$ may also depend on time and on the agent; here for simplicity, we restrict it as state-dependent.

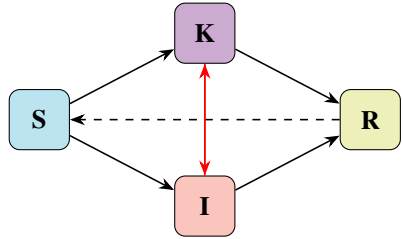
\begin{figure}[h]
    \centering
\begin{tikzpicture}[
  comp/.style = {
    draw,
    rounded corners,
    minimum width=0.8cm,
    minimum height=0.8cm,
    font=\bfseries,
    align=center
  },
  S/.style   = {comp, fill=SkyBlue!40},
  I/.style   = {comp, fill=Salmon!50},
  R/.style   = {comp, fill=GreenYellow!50},
  K/.style   = {comp, fill=Purple!40},
  arr/.style = {->, semithick, >=Stealth}
]

  \node[S] at (0,0)    (S) {S};
  \node[R] at (4.4,0)    (R) {R};
  \node[K] at (2.2,1.15)  (K) {K};
  \node[I] at (2.2,-1.15) (I) {I};

  \draw[arr] (S) -- (K);
  \draw[arr] (S) -- (I);
  \draw[arr] (K) -- (R);
  \draw[arr] (I) -- (R);
  \draw[arr,dashed] (R) -- (S);
  \draw[arr,draw=red] (K) -- (I);
  \draw[arr,draw=red] (I) -- (K);
\end{tikzpicture}
\caption{State Transition Flow for SKIR model}
\label{fig:transition}
\end{figure}

\par{\textbf{Game setup. }} Let $T > 0$ be a finite time horizon. We denote the set of $\RR \supset A$-valued admissible controls by $\mathbb{A}$.\footnote{We assume that admissible controls are closed-loop (i.e., function of agent’s own state), Markovian, and square-integrable.} In the limiting case where the number of agents goes to infinity, we focus on a continuum of agents that are indexed by $x\in I := [0,1]$. Consider agent $x \in I$, her control process is denoted by $\alpha^{x} := (\alpha_{t}^{x})_{t \in [0,T]}$ and represents her \textit{communication level} while her state process is denoted by $(X_{t}^{\boldsymbol{\alpha}, x})_{t \in [0,T]}$ and defined on $E$. Agent $x$'s state process is influenced by the control profile of every agent, $\boldsymbol{\alpha}:=(\alpha^{x})_{x\in I}$, which is added as a superscript to emphasize the influence, through agents' interactions on the undirected network. Under the limit case, such network is represented via graphon function $w$, formally defined as $w: I \times I \rightarrow [0,1]$ where $w(x,y)$ represents the social connection strength between agent $x$ and agent $y$. The aggregate of agent $x \in I$, namely the impact that $x$ receives from others towards her own cost and dynamics is given in the general form as
\small
\begin{equation}
    Z_t^{{\bsalpha}, x}=\int_I w(x, y) \mathbb{E}\left[K\left(\alpha_t^y, X_{t}^{{\bsalpha}, y}\right)\right] d y,
\end{equation}
\normalsize
where $K: A \times E \rightarrow \mathbb{R}$ is called \textit{impact function}. In this work, we will introduce two aggregate variables by making them state-specific:
\small
\begin{equation}
    \label{agg-def}
    \left\{\begin{array}{l}Z_{t, \mathrm{K}}^{\boldsymbol{\alpha}, x}=\int_I w(x, y) \int_A \alpha_t^y \rho_t^y(d a, \mathrm{K}) d y, \\[2mm] 
    Z_{t, \mathrm{I}}^{\boldsymbol{\alpha}, x}=\int_I w(x, y) \int_A \alpha_t^y \rho_t^y(d a, \mathrm{I}) d y.\end{array}\right.
\end{equation}
\normalsize
Intuitively, the aggregate pair of agent $x$ is the weighted communication level of her neighbors in the states $\mathrm{K}$ and $\mathrm{I}$.

Agents choose controls to minimize their individual cost functionals that depends on their own control, state, and the weighted interactions with the population via a graphon. The objective functional also depends on an admissible\footnote{Principal's control profile is admissible if it induces at least one graphon game Nash equilibrium (see below for GGNE definition). When multiple GGNEs exist, the principal selects an equilibrium profile at random.} policy $(\phi_t, \psi_t) =: \lambda_t \in \Lambda \subset \mathbb{R}^2$ of the principal, and the set of all principal's admissible controls is denoted $\mathbf{\Lambda}:=\Phi \times \Psi \subset C([0,T];\mathbb{R})^2$. In the rumor propagation example, $\phi_t$ is designed as the policy that affects minor agents' costs directly and $\psi_t$ is designed as the one that affects minor agents' dynamics. Therefore, when population strategy profile $\boldsymbol{\alpha}$ induces aggregates $\boldsymbol{Z}^{\boldsymbol{\alpha},x} = (Z_{t,j}^{\boldsymbol{\alpha}, x})_{j\in\{\mathrm{K}, \mathrm{I}\},t\in[0,T]}$, agent $x$'s total expected cost for choosing control $\sigma = (\sigma_t)_{t\in[0,T]}$ under policy $\lambda$ is written as $\cJ^{x}(\sigma; \boldsymbol{Z}^{\bsalpha,x},\lambda)$, where
\small
\begin{equation}
    \label{eq:total-expected-cost-formula}
    \begin{aligned}
        \cJ^{x}(\sigma; \boldsymbol{Z},\lambda):=\mathbb{E} &
        \left[\int_0^T f^x\left(t, X_t^{\boldsymbol{\alpha}, x}, \boldsymbol{Z}_t, \sigma_t, \lambda_t\right)dt\right.  \\ &\quad\quad\quad\quad\quad\left.+g^x\left(X_T^{\boldsymbol{\alpha}, x}, \boldsymbol{Z}_T, \lambda_T \right)\right].
    \end{aligned}
\end{equation}
\normalsize
The running and terminal costs of agent $x \in I$ are denoted by $f^{x}:[0, T] \times E \times \mathbb{R}^2 \times A \times \bLambda \rightarrow \mathbb{R}$ and $g^{x}:E \times \mathbb{R}^2 \times \bLambda \rightarrow \mathbb{R}$.In rumor propagation case, they are specified as 
\small
\begin{equation}
    \label{playercosts-1P}
    \begin{aligned}
        f\left(t, e, \boldsymbol{z}, \alpha, (\phi,\psi)\right)&=\frac{1}{2}(1-\alpha)^2-\phi\alpha \mathds{1}_{\rmK}(e), \\ g(e, \boldsymbol{z}, (\phi,\psi)) &= 0,
    \end{aligned}
\end{equation}
\normalsize
where $\mathds{1}_{\rmK}(\cdot)$ denotes the indicator of state K. The first term of running cost $f$ is the penalty on excessive effort on communication or the opportunity costs from insufficient communication levels, where we assume the natural communication rate equals to $1$. The second term is a reward offered by the principal for staying in the preferred (K) state. For simplicity, terminal cost are set as 0 for this motivating example; however, it can be easily generalized to include state, aggregate or, policy dependence. The state process of minor agents follows a continuous-time Markov chain (CTMC). Specifically, the transition rate matrix of agent $x \in I$ is 
\small
\begin{equation*}
Q^{x}(\alpha, \boldsymbol{z}, (\phi,\psi)) = 
\begin{blockarray}{cccccc}
& \rmS & \rmK & \rmI & \rmR \\
\begin{block}{c(ccccc)}
  \rmS & \cdots & \beta_{\rmS} \alpha z_{\rmK}+\psi & \beta_{\rmS} \alpha z_{\rmI} & 0 \\
  \rmK & 0 & \cdots & \beta_{\rmK} \alpha z_{\rmI} & \mu_{\rmK} \\
  \rmI & 0 & \beta_{\rmI} \alpha z_{\rmK}+\psi & \cdots & \mu_{\rmI} \\
  \rmR & \eta & 0 & 0 & \cdots \\
\end{block}
\end{blockarray}
\end{equation*}
\normalsize
where $\cdots$ represents the negative of the sum of all other entries in that row to ensure the row sum to be $0$. The principal is affecting the minor agent dynamics through another control $\psi$, in a way of increasing the transition rate into her preferred (K) state. We then define the minor agents' game equilibrium notion.

\begin{define}
    \label{def:ggne}
    Let $\lambda \in \bLambda$ be an admissible policy of the principal. We call a minor agent strategy profile, $\boldsymbol{\alpha}^{\lambda}$, a \textit{graphon game Nash equilibrium} (GGNE) given the policy $\lambda$ if no agent can gain from a unilateral deviation, i.e. $\forall x \in I, \forall \boldsymbol{\sigma} \in \mathbb{A}$:
\small
\begin{equation*}
    \label{nash-eq-1P}    \mathcal{J}^x\big(\boldsymbol{\alpha}^{\lambda,x} ;\boldsymbol{Z}^{\boldsymbol{\alpha}^{\lambda}, x}, \lambda\big) \leq \mathcal{J}^x\big(\boldsymbol{\sigma} ;\boldsymbol{Z}^{\boldsymbol{\alpha}^{\lambda}, x}, \lambda\big) .
\end{equation*}
\normalsize
\end{define}

\subsection{Principal's Problem}
In the previous section, we mentioned the principal shifts the population opinion towards her favor via policy $\lambda = (\phi_t, \psi_t)_{t\in[0,T]}$. This setting is motivated by two main observations: (i) principal can reward agents in real life: such as during COVID-19, the UK and U.S. paid social-media influencers to promote the benefits of NHS test and vaccination, where compensation and privileged access implicitly/explicitly raised the payoff for amplifying the preferred message~\cite{uk_nhs_2020,colorado_vax_2021}. (ii) Principal can impact the transition rates indirectly in real life: such as governments or media outlets may budget advertisement or create policies to indirectly shift transition rates. For example, in the U.S., parties commonly rely on allied outlets to circulate their messages more rapidly (e.g. use of Facebook in 2012 Presidential Campaign~\cite{obama-campaign-2012}); in China, regulators require recommendation algorithms of social platforms to align with values of the ruling party~\cite{china_algo_2022}; and in Turkey, advertising budgets are directed to pro-government outlets and cut from critical outlets~\cite{turkey_news}.

Let $p^{x, \lambda}(t):=\left(p^{x}(t, e)\right)_{t \in [0,T], e \in E}$ be the state distribution flow of agent $x \in I$ with initial state distribution $p_{0}^{x}$ under strategy profile $\boldsymbol{\alpha}$. We argue that this flow with fixed $\lambda \in \boldsymbol{\Lambda}$ solves the Kolmogorov-Fokker-Planck (KFP) forward equation:
\small
\begin{equation}
    \label{eq:forward-kolm}
    \frac{d}{d t} p^{x,\lambda}(t)=p^{x,\lambda}(t) Q^x\left(\alpha_t^{x}, \bZ_{t}^{\bsalpha,x},\lambda_t\right),
\end{equation}
\normalsize
with initial condition $p^{x,\lambda}(0) = p_{0}^{x}$, where $\bZ_{t}^{\bsalpha,x} = (Z^{\bsalpha,x}_{t, \rmK}, Z^{\bsalpha,x}_{t,\rmI})$. Define the average population density flow ${p}_e^{\lambda}(t):=\int_{I} p^{x,\lambda}(t,e) dx, e \in \{\rmK, \rmI\}$, the principal's objective is to minimize the following cost functional over $\lambda=(\lambda_t)_{t\in [0,T]}$:
\small
\begin{equation}
    \label{principalcost-1P}
    \begin{aligned}
    \cJ_{0}(\lambda):=\int_0^T\left[c_{\lambda}\| \lambda_t\|^2-\hat{p}^{\lambda}_{\rmK}(t)+\hat{p}^{\lambda}_{\rmI}(t)\right] d t,
    \end{aligned}
\end{equation}
\normalsize
where $c_{\lambda}>0$ is a scaling constant for the cost of imposing the policy, and $\hat{p}^{\lambda}_{\rmK}$ and $\hat{p}^{\lambda}_{\rmI}$ are the population density flows for state K and I under graphon game Nash equilibrium control given $\lambda$, i.e, $\bsalpha^\lambda$ (see definition~\ref{def:ggne}). We remark that instead of taking it exogenously, $p^{x}_{0}$ could also be a control of the principal, which creates a constrained optimization problem which is beyond the scope of this work. The principal’s objective rewards the proportion of agents in the preferred state (term 2) while penalizes both the proportion in the non-preferred state (term 3) and regulatory effort (term 1). We are now ready to define the equilibrium for the leader–follower setting.
\begin{define}
\label{def:stackelberg_graphon_nash_single}
 We call a principal policy $\lambda^*$ a (single-principal) Stackelberg graphon game equilibrium (SGGE) if $\lambda^*\in \argmin_{\lambda\in\bLambda} \cJ_{0}(\lambda)$.
\end{define}

\begin{remark}
    We emphasize that between the principal and the mean-field population, we are looking for a Stackelberg equilibrium rather than a Nash equilibrium. In the Stackelberg equilibrium, the principal takes the population response (through the population density of states in this particular model of interest) as a function of her own control and then optimizes her own objectives accordingly. While in a Nash setting, the principal treat the population density exogenous/independent of her control, and give her best response. Achieving a Nash equilibrium would then require solving a fixed-point problem between the principal’s and the population’s best-response mappings. A Nash equilibrium in our setting is an uninteresting problem from both application and mathematical perspectives. On the former side, the Stackelberg equilibrium creates an information hierarchy for the principal which is a good fit for public policy or mechanism design applications. On the latter side, if the population density flows are taken exogenous, the best response of the principal is trivially $\lambda_t=0$ for all $t\in[0,T]$.
\end{remark}

\section{Main Theoretical Results}
\label{sec:theory}
Given a fixed flow $\bZ^x= (\bZ^x)_{\rmK,\rmI}$ and policy of principal $\lambda$, we define the value function of agent $x \in I$ in state $e \in E$ at time $t \in [0,T]$ 
\small
\begin{equation}
\begin{aligned}
u^{x,\lambda}(t, e):=&\inf_{\alpha^x \in \mathbb{A}} \mathbb{E}\Bigg[\int_t^T f^x\Big(s, X_s^{x}, \boldsymbol{Z}_{s}^{x}, \alpha_s^x, \lambda_{s}\Big) d s \\
& \quad\quad\quad\quad\quad+g^x\left(X_T^{x}, \boldsymbol{Z}_{T}^{x},\lambda_T\right) \mid X_t^x=e\Bigg] .
\end{aligned}
\end{equation}
\normalsize
We further denote $q_t\left(e, e^{\prime},\alpha,\boldsymbol{z}, \lambda\right)$ as the element of Q-matrix that gives the transition rate from state $e$ to $e'$.

\begin{assumption}
    \label{assu:1}
    \begin{itemize}
        \item[(i).] We assume $\forall x \in I, t \in [0,T]$, agent x's control is in a bounded set, i.e., $\alpha_{t}^{x} \in [0, \bar{A}]$ with a constant $\bar{A} > 0$; and policy of principal $\lambda:[0,T] \rightarrow \mathbb{R}^2$ is continuous with respect to $t$.
        \item[(ii).] We assume that policy of principal $\lambda:[0,T] \rightarrow \mathbb{R}_{+}^2$ is lipschitz continuous with respect to $t$. Writing $\lambda_t = (\phi_t,\psi_t)$, we assume the components are uniformly bounded: $|\phi_t| \leq \bar{\phi}$ and $|\psi_t| \leq \bar{\psi}$ for all $t \in [0,T]$.
    \end{itemize}
\end{assumption}
We characterize GGNE via an FBODE system under fixed policy; for notational simplicity, we drop superscript on $\lambda$.
\begin{theorem}[GGNE Characterization]
    \label{thm:characterization}
    Under Assumption~\ref{assu:1}.(i), given policy of principal, $\lambda=(\phi_t, \psi_t)_{t\in[0,T]} \in \mathbf{\Lambda}$, the graphon game Nash equilibrium control is given as $\hat{\theta}^{x}(t, e) := \hat{\alpha}^{x,\lambda}_t\left(e,\boldsymbol{z}, u^x(t, \cdot)\right)$ for agent $x \in I$, $\forall e \in E \text{ and } t \in [0,T]$, where
    \small
    \begin{equation}
        \label{eq:opt-cons}
        \begin{aligned} & \hat{\theta}^{x}(t, \rmS)=\beta_{\rmS} Z^{x}_{t,\rmK}(u^{x}(t, \rmS)-u^{x}(t, \rmK))\\&\quad\quad\quad\quad\quad\quad+\beta_{\rmS} Z^{x}_{t,\rmI}(u^{x}(t, \rmS)-u^{x}(t, \rmI))+1 \\ & \hat{\theta}^{x}(t, \rmK)=\beta_{\rmK} Z^{x}_{t,\rmI}(u^{x}(t, \rmK)-u^{x}(t, \rmI))+ 1 +\phi_t \\ & \hat{\theta}^{x}(t, \rmI)=\beta_{\rmI} Z^{x}_{t,\rmK}(u^{x}(t, \rmI)-u^{x}(t, \rmK))+1 \\ & \hat{\theta}^{x}(t, \rmR)=1\end{aligned}
    \end{equation}
    \normalsize
    if the couple $(u, p)$ solves the following FBODE system
    \small
    {
        \allowdisplaybreaks
        \begin{align*}
        \dot{p}^{x}(t, e)&=\sum_{e^{\prime} \in E} p^x(t, \cdot) q_t^{x}\left(e^{\prime}, e, \hat{\theta}^x(t, \cdot),Z_{t, \rmK}^{x}, Z_{t, \rmI}^{x},\lambda_t\right), e \in E \\
\dot{u}^{x}(t, \rmS)&=\beta_{\rmS} Z^{x}_{t,\rmK} \hat{\theta}^{x}(t, \rmS)(u^{x}(t, \rmS)-u^{x}(t, \rmK))\\&\quad+\beta_{\rmS} Z^{x}_{t,\rmI} \hat{\theta}^{x}(t, \rmS)(u^{x}(t, \rmS)-u^{x}(t, \rmI)) \\
&\quad+\psi_t(u^{x}(t,\rmS)-u^{x}(t,\rmK))-\frac{1}{2}(1-\hat{\theta}^{x}(t, \rmS))^2 \\
\dot{u}^{x}(t, \rmK)&=\mu_{\rmK}(u^{x}(t, \rmK)-u^{x}(t, \rmR))-\frac{1}{2}(1-\hat{\theta}^{x}(t, \rmK))^2\\&\quad+\beta_{\rmK} \hat{\theta}^{x}(t, \rmK) Z_{t,\rmI}^{x}(u^{x}(t, \rmK)-u^{x}(t, \rmI))+\phi_{t}\hat{\theta}^{x}(t,\rmK) \\
\dot{u}^{x}(t, \rmI)&=\mu_{\rmI}(u^{x}(t, \rmI)-u^{x}(t, \rmR))-\frac{1}{2}(1-\hat{\theta}^{x}(t, \rmI))^2\\&\quad+\beta_{\rmI} \hat{\theta}^{x}(t, \rmI) Z_{t,\rmK}^{x}(u^{x}(t, \rmI)-u^{x}(t, \rmK)) \\
&\quad+\psi_t(u^{x}(t,\rmI)-u^{x}(t,\rmK))\\
\dot{u}^{x}(t, \rmR)&=\eta(u^{x}(t, \rmR)-u^{x}(t, \rmS)) - \frac{1}{2}(1-\hat{\theta}^{x}(t, \rmR))^2 \\
Z_{t, \rmK}^{x}&=\int_I w(x, y) \hat{\theta}^{y}(t,\rmK) p^{y}(t, \rmK) dy\\
        Z_{t, \rmI}^{x}&=\int_I w(x, y) \hat{\theta}^{y}(t,\rmI) p^{y}(t, \rmI) dy\\
     u^x(T, e)&=0, p^x(0, e)=p_0^x(e),\quad\forall e \in E, \quad\forall x\in I.
        \end{align*}
    }
    \normalsize
    
\end{theorem}

\begin{proof}
    We extend the results in~\cite[Section 7.2]{carmona_MFG_booki} and write a finite-state version of the Hamilton–Jacobi–Bellman (HJB) and Kolmogorov–Fokker–Planck (KFP) equations system for a continuum of agents. We first provide optimality condition for communication level of agent $x$ by writing her Hamiltonian
    \small $$\begin{aligned}
& H^x\left(t, e,\boldsymbol{z}, u, \alpha\right) \\
& \quad=\sum_{e^{\prime} \in E} q_t^x\left(e, e^{\prime},\alpha,\boldsymbol{z},\lambda\right) u\left(e^{\prime}\right)+f^x\left(t, e,\boldsymbol{z}, \alpha,\lambda\right).
\end{aligned}$$\normalsize
By strong convexity of the running cost $f^x$ on control, the mapping $\alpha \mapsto H^x\left(t, e,\boldsymbol{z}, u, \alpha\right)$ admits an unique measurable minimizer $\hat{\theta}^{x,\lambda}(t, e)  := \hat{\alpha}^{x,\lambda}_t\left(e,\boldsymbol{z}, u^x(t, \cdot)\right), e\in E$. We use first-order optimality condition to compute them explicitly and derive~\eqref{eq:opt-cons}. By dynamic programming principle of optimal control, the HJB equations can be written as
\small
\begin{equation}
    \label{eq:hjb_abbrev}
    \dot{u}^{x}(t,e)+H\left(t, e,\boldsymbol{z}, u^x(t, \cdot), \hat{\theta}^{x}(t, e) \right) = 0,
\end{equation}
\normalsize
with terminal conditions $u^{x}(T,e)=0, e \in E$. We remark that HJB is a partial differential equation; on the finite state space, it reduces to an ODE system. We plug in using the definition of Hamiltonian to explicitly derive the backward equations of the FBODE system, and couple them with the KFP equation in~\eqref{eq:forward-kolm}. This completes the proof.
\end{proof}

Next, we state the existence of solution for the FBODE system.

\begin{theorem}[FBODE existence]
    \label{thm:existence}
    Under Assumption~\ref{assu:1}.(ii), given principal's policy $\lambda \in \mathbf{\Lambda}$, if the horizon time is short enough such that \small
    \begin{equation}
        \label{eq:short_time_con}
        T \bar{\beta}\Big[\frac{1}{2}\big((\bar{A}-1)^2\vee1\big)+\bar{\phi}\bar{A}\Big]<1,
    \end{equation}
    \normalsize
    there exists a bounded solution $(u(t,e),p(t,e))_{t\in[0,T],e\in E}$ to the continuum FBODE system stated in Theorem~\ref{thm:characterization}.
\end{theorem}

\begin{proof}
    With $\lambda$ fixed, we treat it as an exogenous parameter to the FBODE system. The solution space we will work on is defined, for $C_1 > 0$,  \small$$\begin{aligned}
        \mathcal{K}_{C_1}:= &\Big\{(u, p) \in C\left([0, T] ; L^2(I \times E) \times L^2(I \times E)\right): \\ &\sum_{e \in E} p^x(t, e)=1, p^x(t, e) \geqslant 0 \ \forall e \in E,\|(u, p)\|_{\infty} \leqslant C_1\Big\}.
    \end{aligned}$$\normalsize \\
    \textbf{Step 1. } We show for any process pair $\mathcal{K}_{C_1} \ni (\boldsymbol{u},\boldsymbol{p}) = (u^x(t,e), p^x(t,e))_{x\in I, t \in [0,T], e \in E}$, the aggregate mapping: 
    \small
    $$\begin{aligned} & \mathcal{T}^{(u, p)}\left(\left(Z_{t, \rmK}^x, Z_{t, \rmI}^x\right)_{t \in[0, T], x \in I}\right)= \\ & \left(\int_I w(x, y) \hat{\theta}^y(t, \rmK) p^y(t, \rmK) d y, \int_I w(x, y) \hat{\theta}^y(t, \rmI) p^y(t, \rmI) d y\right),\end{aligned}$$
    \normalsize
    is a contraction under short-time condition in~\eqref{eq:short_time_con}. It is built on a complete normed aggregate space $(\mathcal{Z}, \|\cdot\|_{\mathcal{Z}})$, where \small $$
    \begin{aligned}
\mathcal{Z}
:=&\ \Bigl\{f \in C\bigl([0,T]; L^1(I;\mathbb{R}^2)\bigr): \\
& |f_t^x|_{[1]}\le \bar A \ \text{and}\ |f_t^x|_{[2]}\le \bar A,\ \ t\in[0,T],\ \text{a.e. }x\in I \Bigr\},
\end{aligned}
$$\normalsize and $\|f\|_{\mathcal{Z}}:=\sup _{t \in[0, T]} \int_I|f(t)(x)| d x$. By Banach fixed point theorem, we proved the unique existence of fixed point for $\mathcal{T}$ for any fixed $(\boldsymbol{u},\boldsymbol{p})$, denoted $\boldsymbol{\hat{Z}}$. \\
\textbf{Step 2. } Given $\boldsymbol{\hat{Z}}$ and $\boldsymbol{u}$, we solve KFP to obtain $\hat{\boldsymbol{p}}$. Then, with $\boldsymbol{\hat{Z}}$ and $\hat{\boldsymbol{p}}$, we solve HJB to obtain $\hat{\boldsymbol{u}}$. The existence and uniqueness of solutions in both steps follow from the Cauchy-Lipschitz-Picard theorem~\cite[Theorem 7.3]{brezis2010functional}. Together with step 1, this defines the mapping $\mathcal{O}: (\boldsymbol{u}, \boldsymbol{p})\mapsto (\hat{\boldsymbol{u}}, \hat{\boldsymbol{p}})$ on $\mathcal{K}_{C_1}$ for large enough $C_1$. By Arzela-Ascoli theorem, $\mathcal{O}(\mathcal{K}_{C_1})$ is compact. In order to apply Schauder fixed-point theorem, it left to verify the continuity of $\boldsymbol{\hat{Z}}$, $\boldsymbol{\hat{p}}$ and $\hat{\boldsymbol{u}}$. We show the continuity of $\boldsymbol{\hat{Z}}$ in the main text, and the rest two continuity results hold similarly using the continuity of $\boldsymbol{\hat{Z}}$, the Assumption~\ref{assu:1}, and Gr\"onwall inequality. We establish continuity of $\boldsymbol{\hat{Z}}$ by considering a sequence $(u^{n}, p^n)_{n \in \mathbb{N}}$ and $(u, p)$ in $\mathcal{K}_{C_1}$ such that $\lim _{n \rightarrow \infty}\left(u^n, p^n\right)=(u, p)$. 
Denote the corresponding aggregate operator as $\mathcal{T}^{n}:=\mathcal{T}^{(u^n,p^n)}$ and $\mathcal{T}:=\mathcal{T}^{(u,p)}$; denote the aggregates they induce as $(\hat{Z}^{n, x}_{\rmK}, \hat{Z}^{n, x}_{\rmI})$ and $(\hat{Z}^{x}_{\rmK}, \hat{Z}^{x}_{\rmI})$, $n \in \mathbb{N}$. Let the contraction constant for the aggregate operator from step 1 be $C_Z < 1$. We first have
{\small
\begin{equation*}
    \begin{aligned}
        \|\hat{Z}^n - \hat{Z}\|_{\mathcal{Z}} &= \|\mathcal{T}^n\hat{Z}^n - \mathcal{T}\hat{Z}\|_{\mathcal{Z}} \\ & \leq \|\mathcal{T}^n \hat{Z}^n - \mathcal{T}^n\hat{Z}\|_{\mathcal{Z}} + \|\mathcal{T}^n \hat{Z} - \mathcal{T}\hat{Z}\|_{\mathcal{Z}} \\
        & \leq C_{Z} \|\hat{Z}^n - \hat{Z}\|_{\mathcal{Z}} + \|\mathcal{T}^n \hat{Z} - \mathcal{T}\hat{Z}\|_{\mathcal{Z}}.
    \end{aligned}
\end{equation*}
}
Therefore, there is
{\small
$$\|\hat{Z}^n - \hat{Z}\|_{\mathcal{Z}} \leq \frac{1}{1-C_Z}\|\mathcal{T}^n \hat{Z} - \mathcal{T}\hat{Z}\|_{\mathcal{Z}}.$$}
It suffices to bound the right hand side. Consider state $\rmK$ without loss of generality, we write the control notation back as $\hat{\alpha}$ to include $(Z,u)$ and obtain
{\small
\begin{equation*}
\begin{aligned}
&|(\mathcal{T}^n\hat{Z})_{t,\rmK} - (\mathcal{T}\hat{Z})_{t,\rmK}| \\
&\leq \int_I w(x,y) \Big|\hat{\alpha}^y_t(\rmK,\hat{Z},u_n) p^{n,y}(t,\rmK) \\ &\hspace{4.2cm}-\hat{\alpha}_t^y(\rmK,\hat{Z},u) p^y(t,\rmK)\Big| dy \\
&\leq \int_I \Big[ w(x,y) \big|\hat{\alpha}_t^y(\rmK,\hat{Z},u_n) - \hat{\alpha}^y_t(\rmK,\hat{Z},u)\big| \big|p^{n,y}(t,\rmK)| \\
&\hspace{1.2cm} + w(x,y) \big|\hat{\alpha}_t^y(\rmK,\hat{Z},u)\big| \big|p^{n,y}(t,\rmK) - p^y(t,\rmK)\big| \Big] dy
\end{aligned}
\end{equation*}
}
The bound goes to 0 as $n \rightarrow \infty$, and similarly hold for state I. We conclude with existence of fixed point for $\mathcal{O}$.
\end{proof}

\section{Multi-Principal Model}
\label{sec:model_multi}
Now, we introduce the model with multiple leaders. We assume that there are two principals: Principal $\cK$ and Principal $\cI$ with controls $\lambda_{\cK}:=(\phi_{\cK,t}, \psi_{\cK,t})_{t\in[0,T]} \in \boldsymbol{\Lambda}$ and $\lambda_{\cI}:=(\phi_{\cK,t}, \psi_{\cK,t})_{t\in[0,T]} \in \boldsymbol{\Lambda}$, respectively. Their controls $(\lambda_{\cK}, \lambda_{\cI})$ affects the minor agents' game via cost functions and dynamics, and they compete with each other by incentivizing the population towards the states they prefer, with principal $\cK$ favoring state $\rmK$ and principal $\cI$ favoring state $\rmI$. We will first start by stating the changes needed in the minor agents' model. Since now both principals encourage transitions into their own preferred states, the Q-matrix that describes the state process of agent $x$ becomes
\small
\begin{equation*}
\begin{aligned}
Q^{x}_2(\alpha,\boldsymbol{z},&(\phi,\psi)_{\cK,\cI}) ={} \\
&\begin{blockarray}{cccccc}
& \rmS & \rmK & \rmI & \rmR \\
\begin{block}{c(ccccc)}
  \rmS & \cdots & \beta_{\rmS}\alpha z_{\rmK}+\psi_{\cK} & \beta_{\rmS}\alpha z_{\rmI}+\psi_{\cI} & 0 \\
  \rmK & 0 & \cdots & \beta_{\rmK}\alpha z_{\rmI}+\psi_{\cI} & \mu_{\rmK} \\
  \rmI & 0 & \beta_{\rmI}\alpha z_{\rmK}+\psi_{\cK} & \cdots & \mu_{\rmI} \\
  \rmR & \eta & 0 & 0 & \cdots \\
\end{block}
\end{blockarray}
\end{aligned}
\end{equation*}
\normalsize
Minor agents are rewarded by the two principals when they actively spreading their preferred news. The running and terminal costs of the agent updates as
\small
\begin{equation}
    \label{playercosts-2P}
    \begin{aligned}
        f\left(t, e, \boldsymbol{z}, \alpha, (\phi,\psi)_{\cK,\cI}\right)&=\frac{1}{2}(1-\alpha)^2-\phi_{\cK}\alpha \mathds{1}_{\rmK}(e)-\phi_{\cI}\alpha \mathds{1}_{\rmI}(e), \\ g(e, \boldsymbol{z}, (\phi,\psi)_{\cK,\cI}) &= 0.
    \end{aligned}
\end{equation}
\normalsize
On the other hand, the two principals pay implementation costs associated with their policies and obtain rewards based on the densities of their preferred states under GGNE. Their optimization problems are formulated as
\small
\begin{equation}
    \begin{aligned}
    &\min _{\lambda_{\cK}} \mathcal{J}_1(\lambda_{\cK};\lambda_{\cI}) \text{ for Principal $\cK$}\hskip0.08cm \min _{\lambda_{\cI}} \mathcal{J}_2(\lambda_{\cI};\lambda_{\cK}) \text{ for Principal $\cI$,}\\ &\text{where:}\\&\mathcal{J}_1(\lambda_{\cK};\lambda_{\cI}):=\int_0^T\left[c_{\lambda}\| \lambda_{\cK,t}\|^2-\hat{p}^{(\lambda_{\cK}, \lambda_{\cI})}_{\rmK}(t)+\hat{p}^{(\lambda_{\cK}, \lambda_{\cI})}_{\rmI}(t)\right] d t \\&\mathcal{J}_2(\lambda_{\cI};\lambda_{\cK}):=\int_0^T\left[c_{\lambda}\| \lambda_{\cI,t}\|^2+\hat{p}^{(\lambda_{\cK}, \lambda_{\cI})}_{\rmK}(t)-\hat{p}^{(\lambda_{\cK}, \lambda_{\cI})}_{\rmI}(t)\right] d t
    \end{aligned}
\end{equation}
\normalsize
where $c_{\lambda}>0$ is a scaling constant for the cost of imposing the policy (which can be also chosen differently for principals), and $\hat{p}^{(\lambda_\cK, \lambda_\cI)}_{\rmK}$ and $\hat{p}^{(\lambda_\cK, \lambda_\cI)}_{\rmI}$ are the population density flows for state K and I under graphon game Nash equilibrium control given $\lambda=(\lambda_{\cK}, \lambda_\cI)$, i.e, $\bsalpha^\lambda$ (see Definition~\ref{def:ggne}). We are ready to define the corresponding game equilibrium.
\begin{define}
\label{def:duo-principal-stackelberg}
    We call a duo-principal strategy couple $(\lambda^{*}_{\mathcal{K}}, \lambda^{*}_{\mathcal{I}})$ a \textit{duo-principal Stackelberg graphon game equilibrium} (DSGE) if $\forall \lambda_{1}, \lambda_{2} \in \mathbf{\Lambda}$,
        \small
        \begin{equation*}
            \begin{aligned}
\mathcal{J}_1(\lambda_{\cK}^{*};\lambda_{\cI}^{*})&\leq\mathcal{J}_1(\lambda_{1};\lambda_{\cI}^{*}),\\[1mm]
            \mathcal{J}_2(\lambda_{\cI}^{*};\lambda_{\cK}^{*}) &\leq \mathcal{J}_2(\lambda_{2};\lambda_{\cK}^{*}).\\[1mm]
            \end{aligned}
        \end{equation*}
        \normalsize
\end{define}
We remark that this formulation could naturally be extended to multi-principal ($\geq 2$) case. Building on this formulation, one can follow the proof of Theorem~\ref{thm:characterization} to obtain a FBODE characterization result of the GGNE. A detailed unique existence analysis of the DSGE is left for future work.

\section{Numerical Study}
\label{sec:numerics}
\subsection{Algorithm Description}
In this section, we propose a bi-level algorithm for solving both the single-principal and duo-principal problems. At the minor-player level, we approximate the continuum model by sampling agents from the underlying graphon, thereby obtaining a finite-dimensional FBODE system. This system could be solved via fixed-point iteration, where we iteratively update the aggregates and controls, and solve forward and backward equations. At the principal level, we solve the optimization problem through a discrete search over her strategy space. The technical details are summarized in Algorithm~\ref{algo:minor_agent}, with steps specific to the duo-principal case highlighted in blue. Under the duo-principal setting, the last step is replaced by computing a two-player Nash equilibrium. Details are presented in Algorithm~\ref{algo:finnash}.

{\begin{algorithm}
\caption{\small Minor-agent graphon game (one principal)\label{algo:minor_agent}}

\footnotesize
\textbf{Input:} Terminal time $T$, graphon $w$, discretization level $N$, regulation bound $(\bar{\phi}, \bar{\psi})$, tolerance $\epsilon$, other model coefficients

\textbf{Output:} converged solution pairs $(u^{*}, p^{*})_{[N]}$, principal costs $(\hat{J}_{\cK},{\color{blue}\hat{J_{\cI}}})_{[N]}$, GGNE aggregates $(Z_{\rmK}, Z_{\rmI})_{[N]}$, optimal policy $\lambda^{*}$

\vskip1mm

\begin{algorithmic}[1]
\STATE Initiate grids $G$ by evenly spacing $[0,\bar{\phi}]\times[0,\bar{\psi}]$ with $N$ points
\FOR{$i \gets 1$ \TO $N$ {\color{blue} and $j \gets 1$ \TO $N$}}
\STATE Fix leaders' control $\lambda^{i}:=G^{i}$ {\color{blue} and $\lambda^{j}:=G^{j}$}
\STATE Initialize flows $(u,p)^{(0)}$
\WHILE{{$\|u^{(k)}-u^{(k-1)}\| > \epsilon$} or {$\|p^{(k)}-p^{(k-1)}\| > \epsilon$} (step k)}
    \STATE Compute $(\boldsymbol{Z}_{\rmK}, \boldsymbol{Z}_\rmI)^{(k)}$ (initiate control with $1$ as step $0$)
    \STATE Compute optimal control $\hat{\theta}^{x,{\lambda^{i,\color{blue}j}}}(t, e)^{(k)}$
    \STATE Solve Kolmogorov-Fokker–Planck equation to achieve $p^{(k+1)}$
    \STATE Solve Hamilton–Jacobi–Bellman equation to achieve $u^{(k+1)}$
    \STATE Save variables in steps 6-9 for the next iteration
\ENDWHILE
\STATE Calculate principal costs $(\hat{J}^{\lambda^{i,\color{blue}j}}_{\cK},{\color{blue}\hat{J}^{\lambda^{i,\color{blue}j}}_{\cI}})$
\ENDFOR
\STATE Solve principal's optimization $\lambda^{*}:=\argmin_{\lambda}(\hat{J}^{\lambda^{i}}_{\cK})_{i=1,\ldots,N}$

\RETURN $(u^{*}, p^{*})_{[N]}$, $(\hat{J}_{\mathcal{K}},{\color{blue}\hat{J_{\mathcal{I}}}})_{[N]}$, $(Z_{\rmK}, Z_{\rmI})_{[N]}, \lambda^{*}$
\end{algorithmic}
\end{algorithm}}
{\begin{algorithm}
\caption{\small Principal-level Nash Search (duo) \label{algo:finnash}}

\footnotesize
\textbf{Input:} $N\times N \times 2$ cost matrix \(\cM\) where $\cM_{ij} := \bigl(C_{1ij},\, C_{2ij}\bigr)$

\vskip1mm

\begin{algorithmic}[1]
\FOR{$i \gets 1$ \TO $N$ and $j \gets 1$ \TO $N$}
  \IF{$C_{1ij} \le \min_{k}\{C_{1kj}\}$ \textbf{and} $C_{2ij} \le \min_{l}\{C_{2il}\}$} 
    \STATE save $(i,j,(C_{1ij},\, C_{2ij}))$
\ENDIF
\ENDFOR
\RETURN qualified entries $\{(i,j,(C_{1ij},\, C_{2ij}))\}$ for equilibria
\end{algorithmic}
\end{algorithm}}

\subsection{Simulation and Real Application}
We present two experiments. The first one is a simulation on a power-law graphon setting, which could naturally extend to different graphon settings. The second focuses on piecewise constan graphon calibrated with real parameters.
\par{\textbf{Simulation Study. }} We consider a power-law graphon of the form $w(x, y)=c(x y)^{-\alpha}$ and specialize to the simple case $c= 1, \alpha=-1$. We then draw $n=50$ agent indices uniformly and independently from $[0,1]$, and order them to form the index set $\mathcal{N}=\left(i_{(1)}, i_{(2)}, \ldots, i_{(n)}\right)$. By construction, agent $i_{(1)}$ exhibits the weakest connectivity with others, while agent $i_{(n)}$ has the strongest, analogous to peripheral and highly influential individuals in real social networks. We implement our model on this sampled graph, choosing parameters symmetrically for states K and I (e.g., $\beta_{\rmS} = \beta_{\rmK}=\beta_{\rmI}, \mu_{\rmK}=\mu_{\rmI}$). We adopt time-independent policies for the principal for computational efficiency; one could instead consider using ML–based approaches for improved solution~\cite{gdml_ml_stackelberg}.

In Figure~\ref{fig:p1f1}, we plot the state distribution flow and aggregates for all 50 sampled agents in the absence of any regulation i.e., $\lambda$ is set equal to $\boldsymbol{0}$ exogenously; the color of the curves get darker with higher agent index, indicating one's stronger network connectivity with others. Agents with weak connectivity exhibit immediate declines in both $K$ and $I$ mass, effectively remaining uninformed and largely insensitive to peer influence. Figure~\ref{fig:p1f2} compares the highest- and lowest-connectivity agents under the \textit{optimized policy} (i.e., Stackelberg equilibrium policy) versus \textit{no regulation}. The imposed policy largely impacts individual state flow in motivating the spread of preferred news (K), even for low connectivity agents.

\begin{figure}[t]
    \centering
    \includegraphics[width=1.0\linewidth]{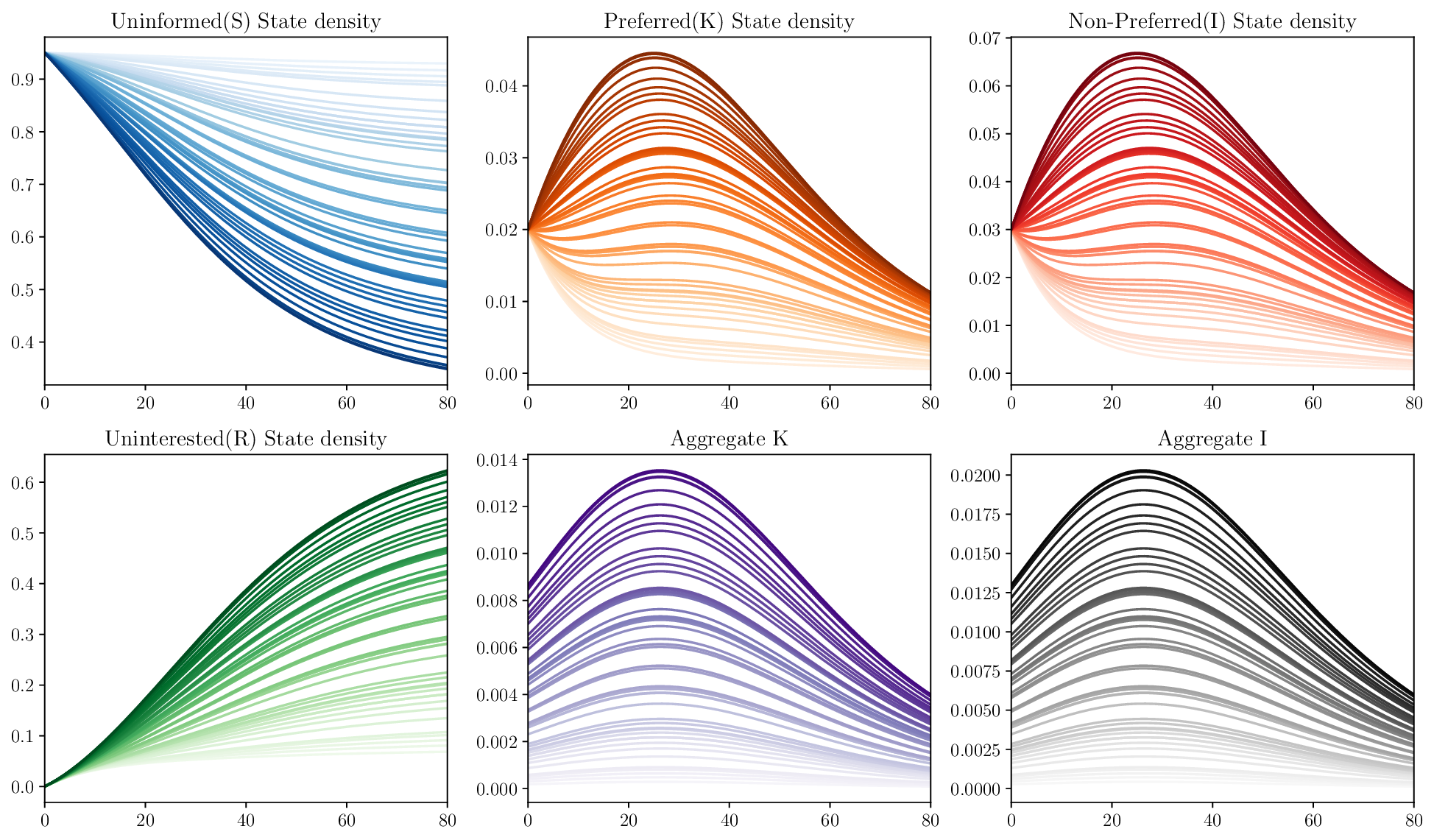}
    \caption{\small Sampled agents' densities and aggregates on the power-law graph. The color of the lines get darker with higher agent index.}
    \label{fig:p1f1}
\end{figure}

\begin{figure}[t]
    \centering
    \includegraphics[width=1.0\linewidth]{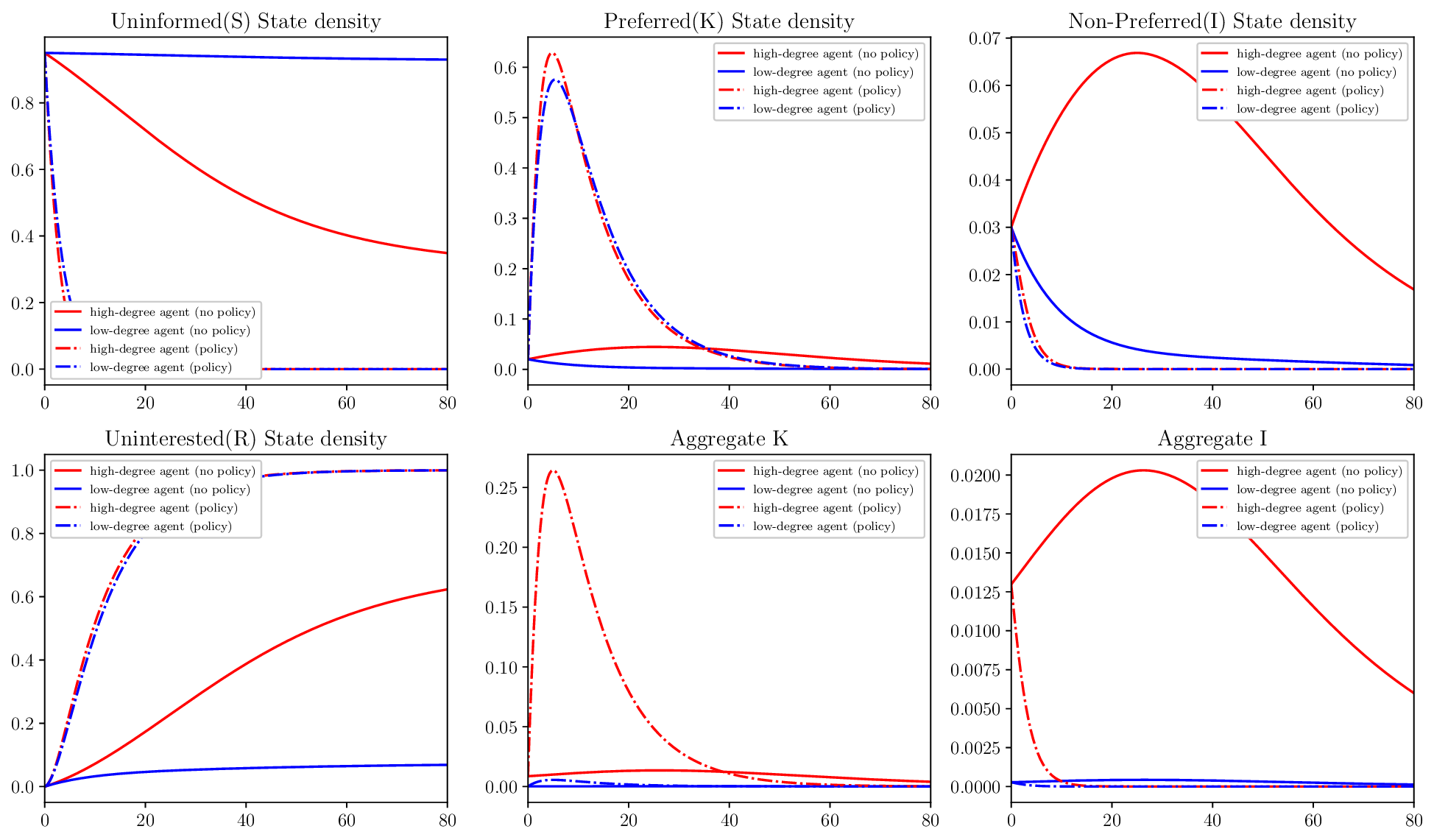}
    \caption{\small Socially (most) active/inactive agent under policy/no policy.}
    \label{fig:p1f2}
\end{figure}

\par{\textbf{Real Application. }}
We extend the first setting from~\cite{rumor_skir} and simulate age-grouped rumor propagation in social networks. We use \textit{piecewise constant graphon} that divides agent indices $[0,1]$ into $K$ intervals with lengths $m^{1}, \cdots, m^{K}$. We assume agents within the same age group are indistinguishable i.e., for $i \in [K]$, consider $x$ and $x'$ from the same group $k$, then $w(x,y) = w(x',y)$, $\forall y \in I$.  In parameter tables~\ref{tab:e1graphon},~\ref{tab:e1param}, we set graphon connectivity to a high level to reflect the rapid evolution of internet-based social networks, and calibrate model parameters from social-platform reports. At time $0$, we set 95\% of the population to be uninformed (S), 2\% to be spreading preferred news (K), and 3\% spreading non-preferred news (I). In this experiment we set $K=4$ to represent 18-29, 30-49, 50-64, 65+ age groups.

\begin{figure}[t]
    \centering
    \includegraphics[width=1.0\linewidth]{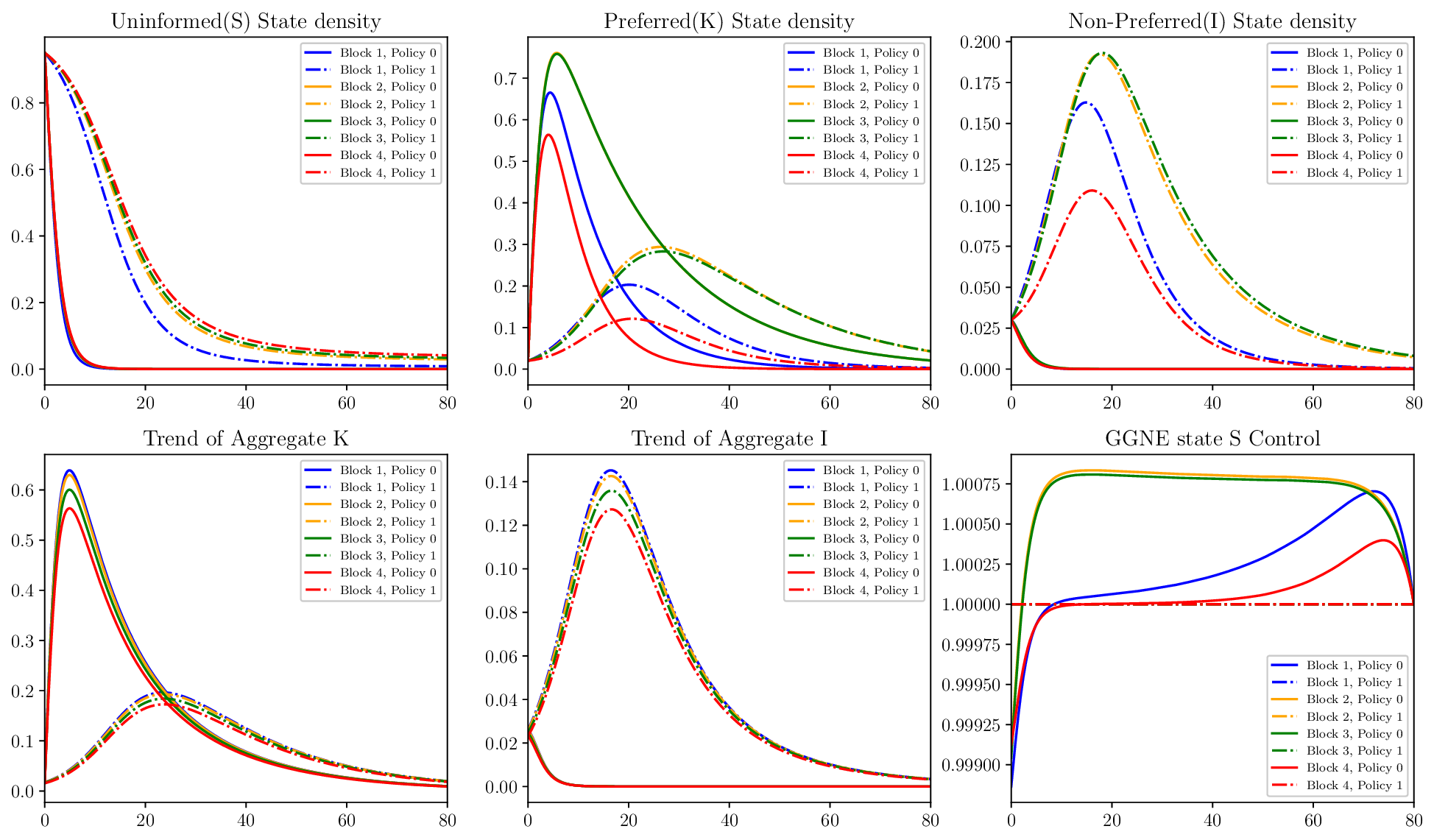}
    \caption{\small Stackelberg Equilibrium Policy (0, solid lines) vs No-Regulation (1, dashed lines) on age-groups.}
    \label{fig:p2f1}
\end{figure}

\begin{figure}[t]
    \centering
    \includegraphics[width=1.0\linewidth]{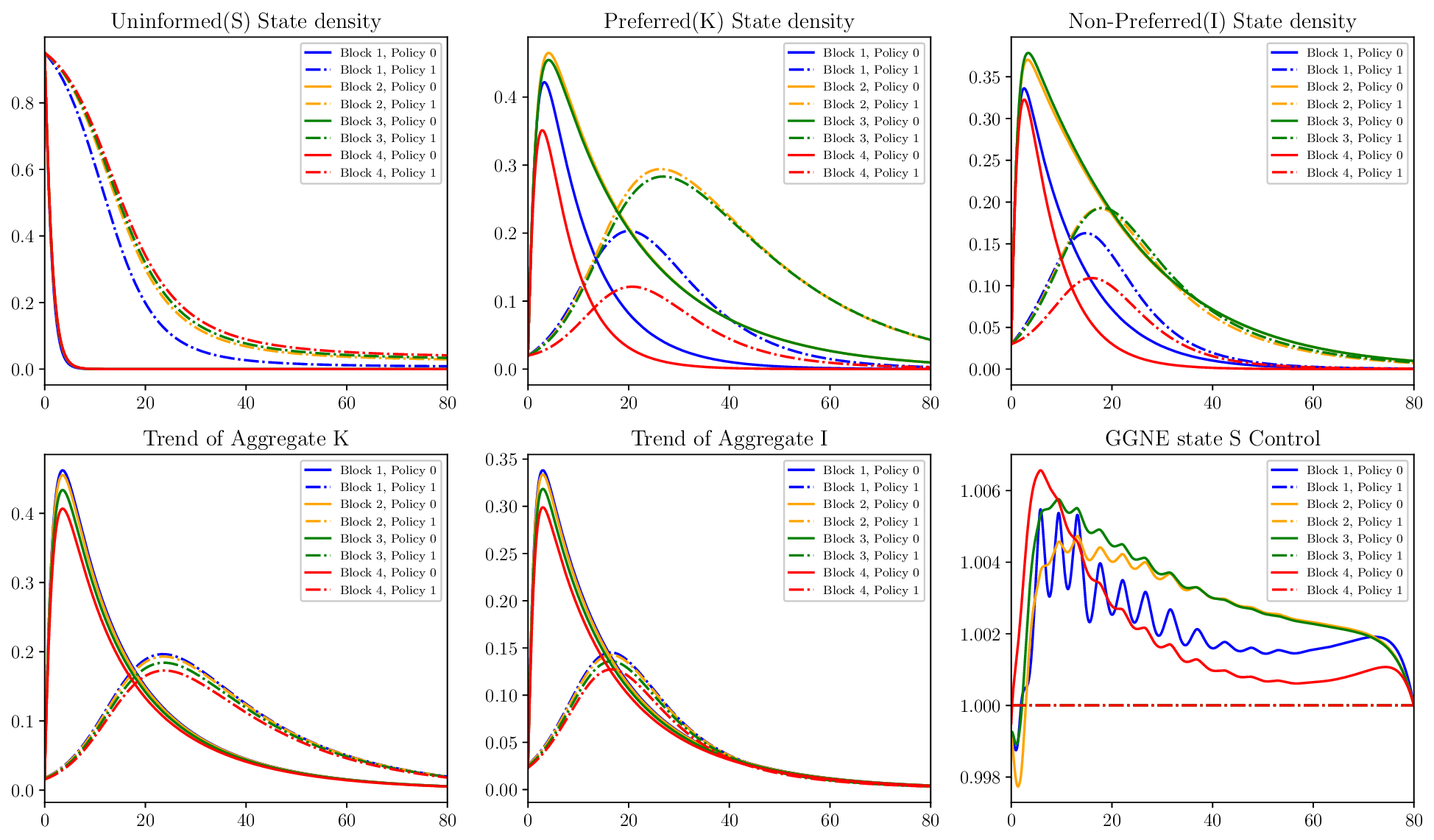}
    \caption{\small Duo-Principal Stackelberg Equilibrium (0, solid lines) vs No-Regulation (1, dashed lines) on age-groups.}
    \label{fig:p2f2}
\end{figure}

Figure~\ref{fig:p2f1} compares the \textit{optimized policy} (i.e., Stackelberg equilibrium policy) with the \textit{no-regulation} baseline i.e., $\lambda$ is set equal to 0 exogenously. The optimized policy induces an early surge in the density of preferred state K, which is later overtaken due to a moderately high forgetting rate. Nevertheless, the spread of non-preferred news gets effectively controlled for all time horizon. Figure~\ref{fig:p2f2} presents the duo-principal case, contrasting the principals' Stackelberg equilibrium outcome with the no-regulation baseline. Both opinions (K and I) peak immediately after time $0$, and this implies that competition between principals can trigger abrupt opinion divisions within the population.

\small
\begin{table}[ht]
\centering
\begin{tabular}{|c|c|c|c|c|}
\hline
Age & 18-29 & 30-49 & 50-64 & 65+ \\ \hline
18-29 & 1 & 0.9 & 0.8 & 0.7 \\ \hline
30-49 & 0.9 & 0.9 & 0.8 & 0.8 \\ \hline
50-64 & 0.8 & 0.8 & 0.9 & 0.8 \\ \hline
65+ & 0.7 & 0.8 & 0.8 & 0.8 \\ \hline
\end{tabular}
\caption{\small Piecewise Constan Graphon}
\label{tab:e1graphon}
\end{table}
\normalsize
\small
\begin{table}[ht]
\centering
\begin{tabular}{|c|c|c|c|c|c|c|}
\hline
\small Age & $\beta_{\textbf{S}}$ & $\beta_{\textbf{K}}$ & $\beta_{\textbf{I}}$ & $\mu_{\textbf{K}}$ & $\mu_{\textbf{I}}$ & $\eta$ \\ \hline
\small 18-29 & 0.4 & 0.5 & 0.75 & 0.1 & 0.1 & 0 \\ \hline
\small 30-49 & 0.3 & 0.42 & 0.62 & 0.05 & 0.05 & 0 \\ \hline
\small 50-64 & 0.3 & 0.32 & 0.48 & 0.05 & 0.05 & 0 \\ \hline
\small 65+ & 0.3 & 0.2 & 0.3 & 0.15 & 0.15 & 0 \\ \hline
\end{tabular}
\caption{\small Experiment Coefficients}
\label{tab:e1param}
\end{table}
\normalsize

\section{Conclusion and Future Work}
\label{sec:conclusion-future}
In this paper, we study the optimal control of rumor propagation in a large population in a game setting by using Stackelberg graphon games. We first introduce the graphon model of the individuals in the population under the policies of a single principal (i.e., regulator) who has \textit{preferred} news. Agents in the population control their jumps between finite number of states that represent their knowledge about the preferred or non-preferred news by choosing their communication level to minimize their own individual costs. Their state dynamics and their cost functionals are affected by the principal's policy choices. Later we introduce the principal's optimization problem which includes wider spread of preferred news and containment of non-preferred news as well as effort costs regarding the policy levels. We prove the existence of the graphon game equilibrium and also provide an extension to multi-leader Stackelberg equilibrium case with competing leaders. We provide numerical algorithm to solve (multi-leader) Stackelberg graphon games and analyze the experimental results with two different graphon choices which are piecewise-constant graphon and power law graphon. For future work, we plan to extend the existence uniqueness results for the Stackelberg graphon game.

\bibliographystyle{ieeetr}
\bibliography{ref.bib}

\section*{Appendix}
\subsection{Finite Player Formulation} \label{sec:finite_player}
To better motivate the study of the graphon game framework, we present the corresponding finite-player game on weighted network. Consider a population of $N \in \mathbb{N}$ agents index by $[N] := \{1,2,\ldots,N\}$. Their interactions are represented by an undirected weighted dense graph, with adjacency matrix $w:[N]\times[N] \rightarrow [0,1]$. For each pair of agents $i,j \in [N]$, $w_{ij}$ is the connection strength between them. Each agent $j \in [N]$ choose her control variable $\alpha_{t}^{j}$, namely her socializing level at time $t \in [0,T]$. Her state process is denoted by $(X_{t}^{j,N})_{t \in [0,T]}$, which evolves as a continuous-time Markov chain on finite state space $\{\rmS, \rmK, \rmI, \rmR\}$, corresponding to \textit{uninformed, spreading news K, spreading news I, uninterested}. Agent $j$ is also impacted by her neighbors who are in states $\rmK$ and $\rmI$ via their controls. In particular, let 
\small
\begin{equation}
    \label{eq:finite_player_Z}
    \begin{aligned}
        Z_{t, \rmK}^{j, N} &:= \frac{1}{N} \sum_{i=1}^N w_{i j} \cdot \alpha_{t}^{i} \cdot \mathds{1}_{\rmK}\big(X_{t}^{i,N}\big), \\ Z_{t, \rmI}^{j, N} &:= \frac{1}{N} \sum_{i=1}^N w_{i j} \cdot \alpha_{t}^{i} \cdot \mathds{1}_{\rmI}\big(X_{t}^{i,N}\big),
    \end{aligned}
\end{equation}
\normalsize
where $(Z_{t, \rmK}^{j, N}, Z_{t, \rmI}^{j, N})$ are defined as the aggregate variables that capture the influence of agent $j$'s neighbors who are actively spreading news at time $t$. In this way, the transition rate of agent $j$ from state S to K, as an example, is proportional to both her own control and K aggregate, i.e. $\propto \alpha_t^j Z_{t, \rmK}^{j, N}$. Denote the control profile of all $N$ agents by $\boldsymbol{\alpha}^{N} := \{\alpha_t^{1}, \ldots, \alpha_t^{N}\}_{t\in[0,T]}$. Given the controls of others $\boldsymbol{\alpha}^{-j}$, agent $j$'s cost functional that she wants to minimize by choosing her control is written as $\mathcal{J}^{j,N}(\boldsymbol{\alpha}^{j};\boldsymbol{\alpha}^{-j})$ similar to the cost functional in~\eqref{eq:total-expected-cost-formula} where we replace the aggregates with the finite-player versions in~\eqref{eq:finite_player_Z}. We are ready to introduce the finite-player Nash equilibrium on the networked system.
\begin{define}
        The control profile $\boldsymbol{\alpha}^N$ is an $N$-player Nash equilibrium if it is admissible and no player can gain from a unilateral deviation, i.e.,
\small
$$
\mathcal{J}^{j, N}\left(\bsalpha^j ; \bsalpha^{-j}\right) \leq \mathcal{J}^{j, N}\left(\sigma; \bsalpha^{-j}\right), \ \forall \sigma \in \mathbb{A}, \ \forall j \in \llbracket N\rrbracket.
$$
\normalsize
\end{define}

\end{document}